\documentclass[12pt]{article}

\usepackage{amsmath,amssymb,bbm,latexsym, amsfonts, amscd, amsthm, mathrsfs}

\theoremstyle{plain}

\newtheorem{theorem}{Theorem}[section]
\newtheorem{conjecture}[theorem]{Conjecture}
\newtheorem{prop}[theorem]{Proposition}

\theoremstyle{definition}

\newtheorem{remark}[theorem]{Remark}

\newcommand{\mint}{\times\!\!\!\!\!\!\!\int}
\long\def\symbolfootnote[#1]#2{\begingroup
\def\thefootnote{\fnsymbol{footnote}}\footnote[#1]{#2}\endgroup} 

\def\alg{{\mathrm{alg}}}

\def\an{{\mathrm{an}}}
\def\lra{\longrightarrow}

\def\GL{{\bf GL}}

\def\A{\mathbf{A}}

\def\sR{\mathscr{R}}

\def\sgn{\mathrm{sgn}}
\def\N{\mathrm{N}}
\def\1{\mf{1}}

\DeclareMathOperator{\Dist}{Dist}

 \DeclareMathOperator{\Norm}{Norm}
\DeclareMathOperator{\rec}{rec} 
\DeclareMathOperator{\Hom}{Hom} 

\DeclareMathOperator{\sign}{sign}

\DeclareMathOperator{\ord}{ord} 
\DeclareMathOperator{\Id}{Id} 

 \DeclareMathOperator{\real}{Re}

 \DeclareMathOperator{\Frob}{Frob}
\DeclareMathOperator{\Ind}{Ind} 
\DeclareMathOperator{\Ta}{Ta}

\newcommand{\stack}[2]{\genfrac{}{}{0pt}{}{#1}{#2}}

\newcommand{\mf}{\mathfrak }

\newcommand{\mscr}{\mathscr }

\DeclareMathOperator{\Real}{Re}

\def\l{\ell}
\def\wmu{\widetilde{\mu}}
\def\wkappa{\widetilde{\kappa}}
\def\womega{\widetilde{\omega}}
\def\fa{\mathfrak{a}}

\def\fp{\mathfrak{p}}
\def\fq{\mathfrak{q}}

\def\fP{\mathfrak{P}}

\def\Z{\mathbf{Z}}

\def\Q{\mathbf{Q}}
\def\C{\mathbf{C}}

\def\R{\mathbf{R}}

\def\bdf{\begin{defn}}
\def\edf{\end{defn}}

\def\cD{\mathcal{D}}
\def\cO{\mathcal{O}}

\def\cU{\mathcal{U}}

\def\cV{\mathcal{V}}

\def\ff{\mathfrak{f}}
\def\fb{\mathfrak{b}}
\def\fc{\mathfrak{c}}
\def\ep{\epsilon}

\def\barX{\overline{X}}

\def\Gal{{\rm Gal}}

\def\Meas{{\rm Meas}}

\def\cF{{\cal F}}

\def\ram{\text{ram}}

\def\sL{{\mscr L}}

\def\sM{{\mscr M}}
\def\sI{{\mscr I}}
\def\sH{{\mscr H}}

\begin{document}
\title{On the Characteristic Polynomial of the Gross Regulator Matrix}
\author{Samit Dasgupta and Michael Spie\ss \\ 
}

\maketitle

\begin{abstract}  We present a conjectural formula for the principal minors and the characteristic polynomial of Gross's regulator matrix associated to a totally odd character of a totally real field.  The formula is given in terms of the Eisenstein
cocycle, which was defined and studied earlier by the authors and collaborators.  For the determinant of the regulator matrix, our conjecture follows from recent work of Kakde, Ventullo and the first author. For the diagonal entries, our conjecture overlaps with the conjectural formula presented in our prior work.  The intermediate cases are new and provide a refinement of the Gross--Stark conjecture.
\end{abstract}

\tableofcontents
\section{Introduction}

\maketitle

Let $F$ be a totally real field of degree $n$, and let \[ \chi\colon \Gal(\overline{F}/F) \longrightarrow \overline{\Q} \] be a totally odd character.
We fix once and for all a prime number $p$ and embeddings $\overline{\Q} \subset \C$ and $\overline{\Q} \subset \C_p$, so $\chi$ may be viewed as taking values in $\C$ or $\C_p$.  Let $H$ denote the fixed field of the kernel of $\chi$, and let $G = \Gal(H/F)$. As usual we view $\chi$ also as a multiplicative map on the semigroup of integral fractional ideals of $F$ by defining $\chi(\fq) = \chi(\Frob_{\fq})$ if $\fq$ is unramified in $H$ and $\chi(\fq) = 0$ if $\fq$ is ramified in $H$.

The field $H$ is a finite, cyclic, CM extension of $F$.  Let
$S_p$ denote the set of primes of $F$ lying above $p$.
We partition $S_p$ as $R \cup R_0 \cup R_1$, where $R$ denotes the primes split in $H$ (i.e.\ $\chi(\fp) = 1$),
$R_0$ denotes the primes ramified in $H$ (i.e.\ $\chi(\fp) =0$),
 and $R_1$ denotes the remaining primes above $p$. 

Let $S = S_{\ram} \cup S_p$ denote the union of $S_p$ with the set of finite places of $F$ that are ramified in $H$.  The Artin $L$-function
\[ L_S(\chi, s) = \sum_{(\fa, S)= 1} \frac{\chi(\fa)}{\N\fa^{s}} = \prod_{\fq} (1 - \chi(\fq)\N\fq^{-s})^{-1}, \quad \real(s) > 1, \]
has an analytic continuation to the entire complex plane.  If we simply write $L(\chi, s)$ for $L_{S_{\ram}}(\chi, s)$, then
\[ L_S(\chi, s) = \prod_{\fp \in R \cup R_1} (1 - \chi(\fp) \N\fp^{-s}) L(\chi, s). \]
It is known that $L(\chi, 0) \neq 0$.  We therefore find 
\begin{equation} \label{e:ordclassical}
 \ord_{s = 0} L_S(\chi, s) = r_\chi, \quad \text{ where } r_\chi = \# R,
\end{equation}
and furthermore 
\begin{equation}
\label{e:ltclassical}
 \frac{L_S^{(r_\chi)}(\chi, 0)}{r_\chi! L(\chi, 0)} = \sR(\chi) \prod_{\fp \in R_1} (1 - \chi(\fp)), \quad \text{ where } 
 \sR(\chi) = \prod_{\fp \in R} \log \N\fp.
\end{equation}
In this paper we refine Gross's conjectural $p$-adic analogs of (\ref{e:ordclassical}) and (\ref{e:ltclassical}), which we now recall.

Let \[ \omega: \Gal(F(\mu_{2p})/F) \lra (\Z/2p\Z)^* \lra \mu_{2(p-1)} \] denote the Teichm{\"u}ller character.  
There is a $p$-adic $L$-function \[ L_p(\chi\omega, s) \colon \Z_p \lra \C_p \] determined by the interpolation
property
\[ L_p(\chi \omega, k) = L_S(\chi\omega^k, k) \text{ for } k \in \Z^{\le 0}. \]
The existence of this function was proved independently by Deligne--Ribet \cite{dr} and Cassou-Nogu\`es \cite{cn} in the 1970s, and new approaches have been considered recently in \cite{pcsd}, \cite{cdg}, and \cite{spiess}.
Gross proposed the following conjecture regarding the leading term of $L_p(\chi\omega, s)$ at $s=0$.
\begin{conjecture}[Gross]  \label{c:gross} We have
\begin{enumerate} \item  \label{i:ov}
\[ \ord_{s=0} L_p(\chi\omega, s) \ge r_\chi. \] \label{e:g1}
\item \label{i:lt}
\[ \frac{L_p^{(r_\chi)}(\chi, 0)}{r_\chi! L(\chi, 0)} =  \sR_p(\chi) \prod_{\fp \in R_0} (1 - \chi(\fp)), \] where $\sR_p(\chi)$
is a certain regulator of $p$-units of $H$ defined  below. \label{e:g2}
\item $\sR_p(\chi) \neq 0$, so in view of (\ref{e:g1}) and (\ref{e:g2}) we have $\ord_{s=0} L_p(\chi\omega, s) = r_\chi.$
\end{enumerate}
\end{conjecture}

It can be shown that for $p \neq 2$, part (\ref{e:g1}) of Gross's conjecture follows from Wiles' proof of the Main Conjecture of Iwasawa theory.  However,
a more direct analytic proof that holds even for $p=2$ was given recently in \cite{pcsd} and \cite{spiess} (note that both of these latter papers use the general 
cohomological results of \cite{spiesshilb}).

Part (\ref{e:g2}) has been proven recently by the first author in joint work with M.\ Kakde and K.\ Ventullo \cite{dkv} (the case $r_{\chi}=1$ had been settled earlier in \cite{ddp} and \cite{vent}).

The goal of this paper is to study and refine Gross's $p$-adic regulator $\sR_p(\chi)$, which we now define.
For each prime $\fp \in S_p$, consider the group of $\fp$-units 
\[ U_\fp := \{u \in H^*: \ord_\fP(u) = 0 \text{ for all } \fP \nmid \fp\}. \]
Write
\begin{align*}
 U_{\fp, \chi} :=& \ (U_\fp \otimes \overline{\Q})^{\chi^{-1}} \\
  =& \ \{u \in U_\fp \otimes \overline{\Q}: \sigma(u) = u \otimes \chi^{-1}(\sigma) \text{ for all } \sigma \in G \} .
  \end{align*}
The Galois equivariant form of Dirichlet's unit theorem implies that 
\[ \dim_{\overline{\Q}} U_{\fp, \chi} = \begin{cases}
 1 & \text{ if } \fp \in R, \\
 0 & \text{ otherwise.}
 \end{cases} \]
 For $\fp \in R$, we let $u_{\fp, \chi}$ denote any generator (i.e.\ non-zero element) of $U_{\fp, \chi}$.
Consider the continuous homomorphisms
\begin{equation}
\begin{aligned}  \label{e:olpdef}
 o_\fp:= \ord_{\fp} \colon& \ F_\fp^* \longrightarrow \Z \\
\ell_\fp:= \log_p \circ  \Norm_{F_{\fp}/\Q_p} \colon& \ F_\fp^* \longrightarrow \Z_p.
 \end{aligned}
\end{equation}
 Suppose we choose for each $\fp \in R$, a prime $\fP_\fp$ of $H$ lying above $\fp$.
 Then for $\fp, \fq \in R$, via
 \[ U_\fp \subset H \subset H_{\fP_\fq} \cong F_\fq, \]
 we can evaluate $o_\fq$ and $\ell_\fq$ on elements of $U_\fp$, and extend by
 linearity to maps \[ o_\fq, \ell_\fq: U_{\fp, \chi} \longrightarrow \C_p. \]  (Of course, $o_\fq(U_{\fp, \chi}) =0$
 for $\fq \neq \fp$.)
 Define
\[ \sL_{\alg}(\chi)_{ \fp, \fq} = - \frac{\ell_\fq(u_{\fp, \chi})}{o_{\fp}(u_{\fp, \chi})},
\]
which is clearly independent of the choice of $u_{\fp, \chi} \in U_{\fp, \chi}$.
 Gross's regulator is the determinant of the $r_\chi \times r_\chi$ matrix whose entries are given by these values:
\[ \sR_p(\chi) := \det(\sM_p(\chi)), \quad \text{ where } \sM_p(\chi) :=  \left(\sL_{\alg}(\chi)_{ \fp, \fq}\right)_{\fp, \fq \in R}. \]

The functions $o_\fq$ and $\ell_\fq$ evaluated on $U_{\fp, \chi}$ depend on the choice of prime $\fP_\fq$ above $\fq$
used to embed $U_{\fp}$ into $F_\fq^*$.  If $\fP_\fq$ is replaced by $\fP_\fq^\sigma$ for $\sigma \in \Gal(H/F)$, then these functions are scaled by
$\chi(\sigma)$.  Accordingly, this change scales the $\fq$th row of $\sM_p(\chi)$ by $\chi(\sigma)^{-1}$ and the $\fq$th column by $\chi(\sigma)$.  In particular, the diagonal entries $\sL_{\alg}(\chi)_{ \fp, \fp}$ are independent of choices, as is the
regulator $\sR_p(\chi)$ and the characteristic polynomial of $\sM_p(\chi)$.  More generally, for any subset $J \subset R$, the principal minor of $\sM_p(\chi)$ corresponding to $J$ defined by
\[ \sR_p(\chi)_J :=   \det\left(\sL_{\alg}(\chi)_{ \fp, \fq}\right)_{\fp, \fq \in J} \]
is independent of choices.  Note that  the characteristic polynomial of a matrix can by expressed simply in terms of the principal minors by
\[
\det(t \cdot \mathbbm{1}_r - \sM_p(\chi)) = \sum_{k=0}^r t^k (-1)^{r-k} \sum_{\stack{J\subset R}{\#J=r-k}}\sR_p(\chi)_J
\]
where $r = r_\chi = \#R$. In this paper we present a conjectural formula for the individual $\sR_p(\chi)_J$ as well as for the characteristic polynomial of $\sM_p(\chi)$ in terms of
purely analytic data depending on $F$ (i.e.\ not depending on knowledge of the algebraic group of $\fp$-units in the extension $H$). 

Our conjectural formula is given in terms of the Eisenstein cocycle, which was defined and studied in \cite{pcsd}, \cite{cdg}, and \cite{spiess}.
For the purposes of this paper, we describe a simplified version of the cocycle, which is a certain group cohomology class
\[ \kappa_\chi \in H^{n-1}(E_R^*, \Meas(F_R, K)). \]
Here $E_R^*$ denotes the rank $n + r - 1$
group of totally positive $R$-units in $F$, 
\[ F_R := \prod_{\fp \in R} F_{\fp}, \]
and $K$ is a finite extension of $\Q_p$ containing the values of the character $\chi$.
 The group $\Meas(F_R, K)$
denotes the $K$-vector space of $K$-valued measures on $F_R$ (i.e.\ $p$-adically bounded linear forms $C_c(F_R, K)\lra K$ where $C_c(F_R, K)$ is the $K$-vector space of compactly supported continuous functions from $F_R$ to $K$). The space  $C_c(F_R, K)$ is endowed with an $F_R^*$-action by \[ (gf)(x):= f(g^{-1}x), \] which induces an action of $E_R^*$ via the diagonal embedding $E_R^* \subset F_R^*$.
This in turn induces an action of $E_R^*$ on the dual $C_c(F_R, K)^\vee$ and its subspace $\Meas(F_R, K)$. 

There are several constructions of the Eisenstein cocycle; in this paper we describe what is perhaps the simplest, in terms of Shintani cones.
In fact, to define the Eisenstein cocycle one must introduce a certain auxiliary prime $\lambda$ of $F$ and use it to employ a smoothing operation.
For simplicity in this introduction, we suppress the prime $\lambda$ from the notation.

Let  $J \subset R$ be a subset.  In \S \ref{section:mainconj} we describe how to define two $r$-cocycles \[ c_{\ell, J}, c_o \in H^r(E_R^*, C_c(F_R, K)). \]
 The subscripts $\ell$ and $o$ refer to the homomorphisms $\ell_\fp$ and $o_\fp$ in (\ref{e:olpdef}) used to define the cocycles for the primes $\fp \in J$.  Let
 \begin{equation} \label{e:vtdef}
  \vartheta \in H_{n+r-1}(E_R^*, \Z) \cong \Z \end{equation} be a generator. We
 define a constant 
\begin{equation} \label{e:rpandef1}
 \sR_p(\chi)_{J, \an} := (-1)^{\# J} \,\frac{c_{\ell, J} \cap (\kappa_\chi \cap \vartheta)}
 {c_o \cap (\kappa_\chi \cap \vartheta)} \in K.
 \end{equation}
The denominator of (\ref{e:rpandef1}) is non-zero because it can be shown 
to equal $L(\chi, 0) \prod_{\fp \in R_0} (1 - \chi(\fp))$ up to sign.
   We propose:
 \begin{conjecture}  \label{c:main} For each subset $J \subset R$, we have $\sR_p(\chi)_J = \sR_p(\chi)_{J, \an}$.
 \end{conjecture}

More generally for $t\in K$ we define a class $c_{to+\l} \in H^r(E_R^*, C_c(F_R, K))$ and propose 
\begin{conjecture}  \label{c:mainchar} 
For the characteristic polynomial of $\sM_p(\chi)$ we have
\[
\det(t \cdot \mathbbm{1}_r - \sM_p(\chi)) = \frac{c_{to+\ell} \cap (\kappa_\chi \cap \vartheta)}
 {c_{o,R} \cap (\kappa_\chi \cap \vartheta)}.
\]
\end{conjecture} 

Conjecture~\ref{c:mainchar} follows from  Conjecture~\ref{c:main} but is slightly weaker (see section \S \ref{section:mainconj}). 
The following results provide the main theoretical justification for our conjecture.
 \begin{theorem}   For $J = R$, we have \[ \sR_p(\chi)_{J, \an} =  \frac{L_p^{(r)}(\chi, 0)}{r! L(\chi, 0) \prod_{\fp \in R_0} (1 - \chi(\fp))}. \]  \end{theorem}
 
Hence Conjecture~\ref{c:main} for $J=R$ is equivalent to part (2) of Conjecture~\ref{c:gross}, whence it holds by \cite{dkv}.
We also consider the other extremal case $\#J = 1$.  The following result is proved for $n=2$, but we suspect that it can be shown in general.

 \begin{theorem}  Let $n=2$ and let $J = \{\fp\}$.  Then Conjecture~\ref{c:main} agrees with the conjectural formula for $\sL_{\alg}(\chi)_{\fp, \fp}$ proposed in
 \cite[Conjecture 3.21]{das}.
 \end{theorem}

For $1 < \# J < r$, Conjecture~\ref{c:main} is a new generalization of Gross's Conjecture.

\section{The Eisenstein Cocycle} \label{s:eis}

To define the Eisenstein cocycle, we first fix an ordering of the real places of $F$, yielding an embedding $F \subset \R^n$.
The group $(\R^*)^n$, and hence $F^* \subset (\R^*)^n$, acts on $\R^n$ by componentwise multiplication.
Given linearly independent vectors $x_1, \dotsc, x_m \in (\R^{>0})^n$, we define the simplicial cone
 \begin{equation} \label{e:conedef}
  C(x_1, \dotsc, x_m) = \{ t_1 x_1 + \cdots + t_m x_m: t_i \in \R^{>0} \} \subset (\R^{>0})^n. 
  \end{equation}
The  vector $Q = (1,0,0, \dots, 0)$  has the property that
its ray (i.e.\ its set of $\R^{>0}$ multiples) is preserved by the action of $(\R^{>0})^n$.  We  define $C^*(x_1, \dotsc, x_n)$ to be the union of 
$C(x_1, \dotsc, x_n)$ with the boundary faces that are brought into the interior of the cone by a small perturbation by $Q$, i.e.~the set whose characteristic function is defined by
\[ 1_{C^*(x_1, \dotsc, x_n)}(x) := \lim_{h \rightarrow 0^+} 1_{C(x_1, \dots, x_n)}(x + hQ). \]

Let $\cO_{F, R}$ denote the ring of $R$-integers of $F$.  For any fractional ideal $\fb \subset F$ relatively prime to $S$, we let
$\fb_{R} = \fb \otimes_{\cO_F} \cO_{F, R}$ denote the $\cO_{F, R}$-module generated by $\fb$.
Let \[ U \subset F_R := \prod_{\fp \in R} F_{\fp} \] be a compact open subset.
Let $C$ be any union of simplicial cones in $ (\R^{>0})^n$.
For $s \in \C$ with $\Real(s) > 1$, consider the Shintani $L$-function
\begin{equation} \label{e:shinL}
 L(C, \chi, \fb, U, s) := \sum_{\stack{\xi \in C \cap \fb_{R}^{-1}, \ \xi \in U}{(\xi, S \backslash R) = 1}} \frac{\chi((\xi))}{\N\xi^s}. 
 \end{equation}
Here $\chi((\xi))$ denotes $\chi$ evaluated on the image of the principal ideal $(\xi)$ under the Artin reciprocity map for the extension $H/F$.
The set $\fb_{R} \cap U$ can be written as the disjoint union of translates of fractional ideals of $F$, which are lattices in $\R^n$.
Shintani proved that the $L$-function defined in (\ref{e:shinL}) has a meromorphic continuation to $\C$, and that its values at nonpositive integers lie in the cyclotomic field $k$ generated by the values of $\chi$.
Furthermore, for $\chi$, $C$, $\fb,$ and $s$ fixed, it is clear that the values $L(C, \chi, \fb, U, s)$ form a distribution on $F_R$ in the sense that
for disjoint compact opens $U_1, U_2 \subset F_R$, we have \[ L(C, \chi, \fb, U_1 \cup U_2, s) =  L(C, \chi, \fb, U_1 , s) +  L(C, \chi, \fb, U_2, s). \]
The space of $k$-valued distributions on $F_R$, denoted $\Dist(F_R, k)$, has an action of $F_R^*$ given by
$(x \cdot \mu)(U) = \mu(x^{-1} U).$   As in the introduction let $E_R^*$ denote the group of totally positive units in $\cO_R$ which we view as a subgroup of $(\R^{>0})^n$.

The following proposition follows directly from \cite[Theorem 1.6]{cdg}.
\begin{prop} \label{p:cocycle}  Let $x_1, \dotsc, x_n \in E_R^*$ and let $x$ denote the $n \times n$ matrix whose columns are the images of the $x_i$ in $(\R^{>0})^n$. 
 For a compact open subset $U \subset F_R$  let 
\[ \mu_{\chi, \fb}(x_1, \dotsc, x_n)(U) := \sgn(x) L( C^*(x_1, \dotsc, x_n), \chi, \fb, U, 0), \]
where $\sgn(x) =  \sign(\det(x))$ if $\det(x)\ne 0$ and $\sgn(x) =0$ otherwise.  Then $\mu_{\chi, \fb}$ is a homogeneous $(n-1)$-cocycle yielding a class $[\mu_{\chi, \fb}] \in H^{n-1}(E_R^*, \Dist(F_R, k)).$
\end{prop}

In order to achieve integral distributions, we introduce a smoothing operation using an auxiliary prime ideal $\lambda$ of $F$.  We assume that $\lambda$ is cyclic in the sense that $\cO_F/\lambda \cong \Z/\l \Z$ for a prime number $ \l \in \Z,$ and we assume that $\l \ge n+2$.  We also assume that no primes in $S$ have residue characteristic equal to $\l$ (in particular $\l \neq p$).  We then define the smoothed Shintani $L$-function
\[ L_\lambda(C, \chi, \fb, U, s) :=  L(C, \chi, \fb \lambda^{-1}, U, s) - \chi(\lambda)\l^{1-s}  L(C, \chi, \fb , U, s). \]
Using ``Cassou--Nogu\`es' trick", it is shown in \cite{cdg} (see also \cite{ds}) that if the generators of the cones comprising $C$ can be chosen to be units at the primes above $\l$, then $L_\lambda(C, \chi, \fb, U, 0) \in \cO_k$.  
  For $x_1, \dotsc, x_n \in E_R^*$ and $U \subset F_R$ an open compact subset let 
\[ \mu_{\chi, \fb, \lambda}(x_1, \dotsc, x_n)(U) := \sgn(x) L_\lambda( C^*(x_1, \dotsc, x_n), \chi, \fb, U, 0). \]
Let $\fP$ be the prime of $k$ above $p$  corresponding to $k \subset \overline{\Q} \subset \C_p$, where the second embedding is the one fixed at the outset of the paper, and let $K=k_{\fP}$. Since $\mu_{\chi, \fb, \lambda}$ is integral, it is in particular $\fP$-adically bounded with values in $\cO_K$ and can therefore be viewed as a $K$-valued measure on $F_R$, i.e.\ as having values in \[ \Meas(F_R, K) := \Hom(C_c(F_R, \Z), \cO_K)\otimes_{\cO_K} K. \] We define 
\begin{equation}
\kappa_{\chi, \fb, \lambda} := [\mu_{\chi, \fb, \lambda}]\in H^{n-1}(E_R^*, \Meas(F_R, K))
\label{e:kapchipart}
\end{equation}
and the Eisenstein cocycle associated to $\chi$ and $\lambda$ by
\begin{equation}
\kappa_{\chi, \lambda}  = \sum_{i=1}^h \chi(\fb_i) \kappa_{\chi, \fb_i, \lambda} \in H^{n-1}(E_R^*, \Meas(F_R, K)).
 \label{e:kapchi}
\end{equation}
Here $\{\fb_1, \dotsc, \fb_h\}$ is a set of integral ideals representing the narrow class group of $\cO_{F,R}$ (i.e.\ the group of fractional ideals of $\cO_{F,R}$ modulo the group of fractional principal ideals generated by totally positive elements of $F$).

We conclude this section by recalling a cap product pairing that can be applied to the Eisenstein cocycle $\kappa_{\chi, \lambda}$.
There is a canonical integration pairing
\begin{align}
\begin{split}
C_c(F_R, K)\times \Meas(F_R, K)\,&\lra\, K, \\
 (f, \mu) & \longmapsto \int_{F_R} f(t) d \mu(t):= \lim_{|| \cV || \rightarrow 0} \sum_{V \in \cV} f(t_V) \mu(V) 
 \label{e:intdef}
\end{split}
\end{align}
where the limit is taken over uniformly finer covers $\cV$ of the support of $f$ by open compacts $V$, and $t_V \in V$ is any element. More generally if $A$ is a locally profinite $K$-algebra (i.e.\ an Iwasawa algebra) we have an $F_R$-equivariant integration pairing 
\begin{align} 
\begin{split}
C_c(F_R, A)\times \Meas(F_R, K)\, &\lra\,A, \\
 (f, \mu) &\longmapsto \int_{F_R} f(t) d \mu(t) \label{e:intdef2}
\end{split}
\end{align}
(see e.g.\ \cite[\S 2]{ds}) where the $F_R^*$-action on $C_c(F_R, A)$ is given by $(x\cdot f)(y) = f(x^{-1}y)$.  For each non-negative integer $m$, the pairing (\ref{e:intdef2}) induces a cap-product pairing
\begin{equation}
\cap: H^{m}(E_R^*, \Meas(F_R, K)) \times H_{m}(E_R^*, C_c(F_R, A)) \,\lra\, A.
\label{e:cap1}
\end{equation}

\section{Conjecture}
\label{section:mainconj}

\subsection{Statement}
We recall the following definition from \cite{spiesshilb, ds}.  Let $\fp \in R$, let $g\colon F_\fp^* \longrightarrow K$ be a continuous homomorphism and let $f\in C_c(F_{\fp}, \Z)$, i.e.\ $f\colon F_{\fp} \lra \Z$ is a locally constant function with compact support. For $a\in F_{\fp}^*$ let $af\in C_c(F_{\fp}, \Z)$ be given by $(af)(x) = f(a^{-1}x)$. Since $(1-a) \cdot f= f-af$ vanishes at $0\in F_{\fp}$, the function \begin{align*}
F_{\fp}^* & \longrightarrow K \\
 x& \longmapsto (f(x)-f(a^{-1}x)) \cdot g(x) \end{align*}
  extends continuously to $F_{\fp}$ hence can be viewed as a function \[ (f-af) \cdot g \colon F_{\fp}\to K. \] The map
 \[ z_{f, g}\colon F_\fp^* \longrightarrow C_c(F_\fp, K) \] given by
 \begin{align} 
 \label{e:zgdef}
  z_{f,g}(a) & = `` (1 - a)(g \cdot f) "  \nonumber \\ & :=  (af) \cdot (g-ag)+ (f-af)\cdot g  \\
  &  = (af) \cdot g(a)  + (f-af) \cdot g \nonumber
 \end{align}
 is an inhomogeneous 1-cocycle. Its class $[z_{f,g}] \in H^1(F_\fp^*, C_c(F_\fp, K))$ depends only on the value of $f$ at $0$, i.e. if $f,f'\in C_c(F_{\fp}, \Z)$ satisfy $f(0)=f'(0)$ then $[z_{f,g}] = [z_{f',g}]$. In particular if we choose $f$ so that $f(0)=1$, e.g.\ $f= 1_{\cO_{\fp}}$ then 
 \[
 c_g := [z_{f,g}] \in H^1(F_\fp^*, C_c(F_\fp, K))
 \]
depends only on $g$. Note that the expression $ (1 - a)(g \cdot f)$ has no literal meaning since the function $g$ does not necessarily extend to a continuous function 
on $F_\fp$ (and for this reason, the cocycle $z_g$ is not necessarily a coboundary); nevertheless this expression provides the intuition for the definition of $z_g$ given by the right side of (\ref{e:zgdef}).

The construction above in particular applies
to the homomorphisms $o_\fp, \l_\fp$ from (\ref{e:olpdef}) and we thus obtain classes
$c_{o_\fp}, c_{\l_{\fp}} \in H^1(F_\fp^*, C_c(F_\fp, K))$  for each $\fp \in R$. 

Recall that $r = r_\chi = \#R$ and that $E_R^*$ denotes the rank $n + r - 1$
group of totally positive $R$-units in $F$. 
 As in (\ref{e:vtdef}), let $\vartheta \in H_{n+r-1}(E_R^*, \Z) \cong \Z$ be a generator (this is well-defined up to sign).  Cap-product with the Eisenstein cocycle yields a class
 \[ \kappa_{\chi, \lambda} \cap \vartheta \in H_r(E_R^*, \Meas(F_R, K)). \]
Label the elements of $R$ by $\fp_1, \fp_2, \dots, \fp_r$ and let $J\subset R$. Define classes 
\[ c_{o}, c_{\l, J} \in H^r( F_R^*, C_c(F_R, K) ) \] by 
\begin{align*}
c_{o} &= c_{o_{\fp_1}}  \cup \cdots \cup c_{o_{\fp_r}}\\
c_{\l, J} &= c_{g_1} \cup \cdots \cup c_{g_r}
\end{align*}
where 
\[
g_i= \left\{\begin{array}{cc} \l_{\fp_i} & \mbox{ if $i\in J$,}\\ o_{\fp_i}& \mbox{if $i\not\in J$.}\end{array}\right.
\]
Here the cup-product is induced by the canonical map
\[
C_c(F_{\fp_1}, K) \otimes \cdots \otimes C_c(F_{\fp_r}, K) \lra C_c(F_R,K)
\]
that sends a  tensor $f_1\otimes \ldots \otimes f_r$ to the function 
\begin{align*}
F_R = \prod_{i=1}^r F_{\fp_i} &\lra K \\
(x_i)_{i=1,\ldots, r}&\longmapsto \prod_{i=1}^r f_i(x_i).
\end{align*}
Define a constant 
\begin{equation} \label{e:rpandef}
 \sR_p(\chi)_{J, \an} := (-1)^{\# J}\, \frac{c_{\l, J} \cap (\kappa_{\chi, \lambda} \cap \vartheta)
 }{c_{o} \cap (\kappa_{\chi, \lambda} \cap \vartheta)} \in K.
 \end{equation}
Here the first cap-product of the numerator and denominator is the pairing (\ref{e:cap1})
for $m=r$ and $A=K$.  We will show in Proposition~\ref{p:denom} below that 
the denominator of (\ref{e:rpandef}) is non-zero and in 
Proposition~\ref{p:kappa} that the constant $\sR_p(\chi)_{J, \an}$ is independent of the auxiliary prime $\lambda$.

We propose:
 \begin{conjecture}  \label{c:main2} For each subset $J \subset R$, we have $\sR_p(\chi)_J = \sR_p(\chi)_{J, \an}$.
 \end{conjecture}
For $t\in K$ we define the class $c_{to+\l} \in H^r(F_\fp^*, C_c(F_\fp, K))$ by 
\[
c_{to+\l} = c_{(to_{\fp_1} + \l_{\fp_1})}  \cup \cdots \cup c_{(to_{\fp_r} + \l_{\fp_r})}
\]
and propose
\begin{conjecture}  \label{c:main2char} The characteristic polynomial of Gross' regulator matrix is given by
\[
\det(t \cdot 1_r - \sM_p(\chi)) = \frac{c_{to+\l} \cap (\kappa_{\chi, \lambda} \cap \vartheta)
 }{c_{o} \cap (\kappa_{\chi, \lambda} \cap \vartheta)}.
\]
 \end{conjecture}
Conjecture \ref{c:main2char} follows from Conjecture \ref{c:main2} since we have
\begin{align*}
c_{to+\l} &=  \sum_{k=0}^r t^k \sum_{\stack{J\subset R}{\#J=r-k}} c_{\l, J},\\
\det(t \cdot \mathbbm{1}_r - \sM_p(\chi)) & =  \sum_{k=0}^r t^k (-1)^{r-k} \sum_{\stack{J\subset R}{\#J=r-k}} \sR_p(\chi)_J.
\end{align*}

\begin{remark}
\rm Instead of considering only the homomorphisms $\l_{\fp}$ in the definition of in Gross' regulator matrix and in the nominator of (\ref{e:rpandef}) one may consider arbitrary continuous homomorphisms $\psi_1: F_{\fp_1}^*\to E, \ldots, \psi_r: F_{\fp_r}^*\to K$. So we conjecture that more generally we have 
\[
\det\left(\left(- \frac{\psi_j(u_{\fp_i, \chi})}{o_{\fp_i}(u_{\fp_i, \chi})}\right)_{i,j=1, \ldots, r}\right) = (-1)^r\, \frac{c_{\psi_{\fp}, \fp\in R} \cap (\kappa_{\chi, \lambda} \cap \vartheta)
 }{c_{o} \cap (\kappa_{\chi, \lambda} \cap \vartheta)}
 \]
where $c_{\psi_{\fp}, \fp\in R} = c_{\psi_1} \cup \cdots \cup c_{\psi_r}$.
\end{remark}

\subsection{Well-formedness of the conjecture}

The following proposition shows that the denominator of (\ref{e:rpandef}) is non-zero.  The result was essentially proved in \cite{ds}, but we explain here how to relate our present notation to the setting of {\em loc.\ cit.}

\begin{prop} 
\label{p:denom}
With the correct choice of sign for $\vartheta$ we have
\[
c_{o} \cap (\kappa_{\chi, \lambda} \cap \vartheta) = (-1)^{r(n-1)} (1 - \chi(\lambda)\l)L(\chi, 0) \prod_{\fp \in R_0} (1 - \chi(\fp)).
\]
\end{prop} 
\begin{proof} Let $\sI_R$ be the group of fractional ideals of $F$ generated by the elements of $R$ and let $\sH_R\subseteq \sI_R$ be the subgroup of principal fractional ideals that have a totally positive generator. Let
$\fc_1,\ldots, \fc_m$ be a system of representatives for $\sI_R/\sH_R$.  Recall from (\ref{e:kapchi}) that  $\{\fb_1, \dotsc, \fb_h\}$ denotes a set of integral ideals representing the narrow class group of $\cO_{F,R}$.
Note that $\{\fb_i\fc_j\mid i=1,\ldots, h, \, j=1,\ldots, m\}$ is a system of representatives for the narrow class group of $F$. 

Let $E^*$ be the group of totally positive units of $\cO_F$. Let $\cD$ denote a signed Shintani domain for the action of $E^*$ on $(\R^{>0})^n$.  This is a finite formal linear combination of simplicial cones (as in (\ref{e:conedef}))
\[ \cD = \sum_{\nu} a_{\nu} C_\nu, \qquad a_\nu \in \Z, \quad C_\nu = \text{simplicial cone,} \]
whose characteristic function $1_\cD := \sum_\nu a_\nu 1_{C_\nu}$ satisfies 
\[ \sum_{\epsilon \in E^*}1_{\cD}( \epsilon \cdot x) = 1 \]
for all $x \in (\R^{>0})^n$. We have (compare e.g.\ \cite[Lemma 5.8]{ds})
\begin{align}
\begin{split}
\label{e:shinlchi}
\sum_{\nu} \sum_{i=1}^h\sum_{ j=1}^m a_\nu \chi(\fb_i) \N(\fb_i\fc_j)^{-s} L_\lambda(C_\nu, \chi, \fb_i, (\fc_j^{-1})^R, s) &=   \\
 (1-\chi(\lambda) \l^{1-s}) L(\chi, s) &\prod_{\fp \in R_0} (1 - \chi(\fp)\N\fp^{-s})
\end{split}
\end{align}
where $(\fc_j^{-1})^R := \prod_{\fp\in R} \fc_j^{-1}\otimes \cO_{\fp}$. Since $(\fc_j^{-1})^R$ is an $E^*$-stable subset of $F_R$ we see that by restricting the cocycle $\mu_{\chi, \fb_i, \lambda}$ to 
$E^*$ and keeping $U=(\fc_j^{-1})^R$ fixed, i.e.\ by setting 
\[ 
\mu_{\chi, \fb_i, \fc_j, \lambda}(x_1, \dotsc, x_n) := \mu_{\chi, \fb_i, \lambda}(x_1, \dotsc, x_n)((\fc_j^{-1})^R)
\]
for $x_1, \ldots, x_n\in E^*$ we obtain a homogeneous $(n-1)$-cocycle yielding a class \[ [\mu_{\chi, \fb_i, \fc_j, \lambda}]\in H^{n-1}(E^*, \cO_K).\]  For the correct choice of the generator $\vartheta_0$ of $H^{n-1}(E^*, \Z)\cong \Z$ we obtain
\begin{equation}
\label{e:shinlchi2}
(\sum_{i=1}^h\sum_{ j=1}^m \chi(\fb_i) [\mu_{\chi, \fb_i, \fc_j, \lambda}])\cap \vartheta_0 =  (1-\chi(\lambda) \l) L(\chi, 0)\prod_{\fp \in R_0} (1 - \chi(\fp)).
\end{equation}
Indeed by \cite[Theorem 1.5]{cdg} the left hand side is equal to the left hand side of (\ref{e:shinlchi}) for $s=0$ and a specific signed Shintani domain $\cD = \sum a_\nu C_\nu$.  

This formula can also be interpreted as follows. Let $\Z[\sI_R]$ be the group ring of $\sI_R$.
Consider the $E_R^*$-equivariant isomorphism
\begin{equation}
\label{e:shinlchi3}
(\Ind_{E^*}^{E_R^*} \Z)^m\lra \Z[\sI_R]
\end{equation}
corresponding by Frobenius reciprocity to the  
$E^*$-equivariant map 
\begin{align*}
\Z^m &\lra \Z[\sI_R] \\
(a_1, \ldots, a_m) &\longmapsto \prod_{j=1}^m \fc_j^{a_j}. 
\end{align*}

 The assignment $\fc\mapsto 1_{\fc^R}$ extends to a homomorphism 
 \begin{equation} \label{e:shinlchi4}
 \delta\colon \Z[\sI_R]\lra C_c(F_R, \Z). 
 \end{equation}  Composing (\ref{e:shinlchi3}) and (\ref{e:shinlchi4}) we obtain an $E_R^*$-equivariant map
 \begin{equation} \label{e:shinlchi5}
 (\Ind_{E^*}^{E_R^*} \Z)^m \lra C_c(F_R, \Z).
 \end{equation}
 
By Shapiro's Lemma the homomorphism (\ref{e:shinlchi5}) induces a homomorphism 
\[
H_{n-1}(E^*, \Z)^m \lra H_{n-1}(E_R^*, C_c(F_R, \Z))
\]
and we denote by $\widetilde{\vartheta}\in H_{n-1}(E_R^*, C_c(F_R, \Z))$ the image of $(\vartheta_0, \ldots, \vartheta_0)\in H_{n-1}(E^*, \Z)^m$. Tracing through the definitions, formula (\ref{e:shinlchi2}) is equivalent to \[ \kappa_{\chi, \lambda} \cap \widetilde{\vartheta} =  (1 - \chi(\lambda)\l) L(\chi, 0)\prod_{\fp \in R_0} (1 - \chi(\fp)). \] 
On the other hand by \cite[Lemma 3.5]{ds} we can choose $\vartheta$ so that $\widetilde{\vartheta} = c_o \cap \vartheta$. It follows that
\[
c_o \cap (\kappa_\chi \cap \vartheta) 
= (-1)^{r(n-1)} \kappa_{\chi}\cap \widetilde{\vartheta} = (-1)^{r(n-1)} (1- \chi(\lambda)\l) L(\chi, 0) \prod_{\fp \in R_0} (1 - \chi(\fp))
\] as desired.
\end{proof}

\begin{prop} 
\label{p:kappa}
The constant $\sR_p(\chi)_{J, \an}$ is independent of the choice of the auxiliary prime $\lambda$.
\end{prop}

\begin{proof} To see this it is useful to work within the adelic framework (compare e.g.\ \cite{ds} or \cite{ds2}).   Before delving into the details, let us summarize the basic idea of the proof.
Let $\lambda'$ be another auxiliary cyclic prime ideal lying above a prime number $\l'\ge n+2$ with $\l\ne \l'$, $S_{\l'}\cap S=\emptyset$.  Replacing $F_R$ with certain adelic spaces, we will describe the construction of   classes 
\[ \tilde{\kappa}_{\lambda}, \tilde{\kappa}_{\lambda'}, \text{ and } \tilde{\kappa}_{\lambda, \lambda'} \]
lying in cohomology groups endowed with an action of the group $\sI^S$ of fractional ideals  of $F$ that are relatively prime to $S$.
It will  follow directly from the definitions that
\begin{equation}
\label{e:changelambda}
  \wkappa_{\lambda} - \l' (\lambda'^{-1} \cdot \wkappa_{\lambda}) = \tilde{\kappa}_{\lambda, \lambda'} = \wkappa_{\lambda'} - \l (\lambda^{-1} \cdot \wkappa_{\lambda'}).
\end{equation}
Under cap product with the classes appearing in the numerator and denominator of  the definition of
 $\sR_p(\chi)_{J, \an}$, the action of $\lambda^{-1}$ is given by multiplication by $\chi(\lambda)$.
Therefore, replacing $\wkappa_{\lambda}$ by the left side of (\ref{e:changelambda}) scales the numerator and denominator of $\sR_p(\chi)_{J, \an}$ each by the constant $1 - \l' \chi(\lambda')$, and leaves the ratio unchanged.  The desired result then follows from (\ref{e:changelambda}).  Let us now carry out the details of this construction.

Let $\A_f$ be the ring of finite adeles of $F$.  For a finite set of finite places $\Sigma$ of $F$ we put \[ \A\!_f^{\Sigma} = \sideset{}{'}\prod_{v\not\in \Sigma\cup S_{\infty}} F_v,  \qquad U^{\Sigma}=\prod_{v\not\in \Sigma\cup S_{\infty}} \cO_v^* \subset (\A\!_f^{\Sigma})^*. \]

Let $a^SU^S\in (\A\!_f^{S})^*/U^S$ and let \[ \fa = F\cap a^S\prod_{v\not\in S\cup S_{\infty}} \cO_v \] be the fractional $\cO_{F, S}$-ideal attached to the idele $a^S$. For linearly independent elements $x_1, \dotsc, x_n \in F^*_+ \subset (\R^{>0})^n$  and a compact open set
 $U \subset F_R\times F_{S-R}^*$,
the Shintani zeta function 
\[ \zeta(x_1, \dots, x_n; U, \fa, s) = \sum_{\stack{\xi \in C^*(x_1, \dotsc, x_n)}{\xi\in \fa \cap U}} \frac{1}{\N\xi^s}, \qquad \Real(s) > 1
\]
has a meromorphic continuation to $\C$ and its values at non-positive integers are rational. 
As any compact open subset of \[ F_R \times (\A\!_f^{R})^*/U^S = F_R\times F_{S-R}^* \times (\A\!_f^{S})^*/U^S \] can be written as a disjoint union of sets of the form $U\times a^S U^S$ considered above, the assignment 
\[
U\times a^S U^S \longmapsto \sgn(x) \zeta(x_1, \dots, x_n; U, \fa, 0)
\]
defines a $\Q$-valued distribution $\wmu(x_1, \dots, x_n)$ on $F_R \times (\A\!_f^{R})^*/U^S$. Moreover \[ (x_1, \dots, x_n)\longmapsto \wmu(x_1, \dots, x_n) \] is a homogeneous $(n-1)$-cocycle for the group $F_+^*$ of totally positive elements of $F$. 

As before, ``smoothing" with respect to the auxiliary prime $\lambda$ yields an integral valued distribution. For that let $S_\l$ denote the set of primes of $F$ lying above $\l$ and  
let $(F_+^{\l})^* \subset F^*_+ $  be the subgroup of elements with valuation $0$ at every prime in $S_{\l}$. Given a compact open subset \[ U \subset F_R \times (\A\!_f^{R\cup S_{\l}})^*/U^{S\cup S_{\l}} \] we define two associated compact open subsets \[ U_0, U_1 \subset F_R \times (\A\!_f^{R})^*/U^S \] by
\[ U_0 = U \times \prod_{v \mid  \l} \cO_v^*, \qquad U_1 = U \times \varpi_\lambda \cO_\lambda^* \times \prod_{\stack{v \mid  \l}{v \neq \lambda}}\cO_v^*, \]
where $\varpi_{\lambda}$ is a prime element of $\cO_{\lambda}$. For $x_1, \dotsc, x_n \in (F_+^{\l})^*$, define
\[
\wmu_{\lambda}(x_1, \dots, x_n)(U) = \wmu(x_1, \dots, x_n)(U_0) - \l \wmu(x_1, \dots, x_n)(U_1).
\]
Then $(x_1, \dots, x_n)\mapsto \wmu_{\lambda}(x_1, \dots, x_n)$ is a homogeneous $(n-1)$-cocycle for $(F_+^{\l})^*$ with values in \[ \Dist(F_R \times (\A\!_f^{R\cup S_{\l}})^*/U^{S\cup S_{\l}}, \Z)= \Hom(C_c(F_R \times (\A\!_f^{R\cup S_{\l}})^*/U^{S\cup S_{\l}}, \Z), \Z). \] Note that there exists a canonical isomorphism of $F_+^*$-modules
\[
\Ind_{(F_+^{\l})^*}^{F^*_+} C_c(F_R \times (\A\!_f^{R\cup S_{\l}})^*/U^{S\cup S_{\l}})\cong C_c(F_R \times (\A\!_f^{R})^*/U^S)
\]
and hence, by Shapiro's Lemma, an isomorphism
\begin{equation} \label{e:shapiro}
H^{n-1}((F_+^{\l})^*, \Dist(F_R \times (\A\!_f^{R\cup S_{\l}})^*/U^{S\cup S_{\l}}, \Z)) \cong H^{n-1}(F_+^*, \Dist(F_R \times (\A\!_f^R)^*/U^S, \Z)). 
\end{equation}
A $\Z$-valued distribution is of course $p$-adically bounded and hence yields a 
measure.  Therefore by (\ref{e:shapiro}) the cohomology class of the cocycle $\wmu_{\lambda}$ defines an element 
\begin{equation}
\label{e:wkappa}
\wkappa_{\lambda} = \wkappa_{\lambda}^R \in H^{n-1}(F_+^*, \Meas(F_R \times (\A\!_f^R)^*/U^S, K)).
\end{equation}
The canonical action of the finite ideles $\A_f^*$ on $C_c(F_R \times (\A\!_f^R)^*/U^S, \Z)$ induces an action of the idele class group $\A_f^*/F_+^* U^S$ on $H^{n-1}(F_+^*, \Meas(F_R \times (\A\!_f^R)^*/U^S, K))$. In particular we obtain an action of  the group $\sI^S$ of fractional ideals  of $F$ that are relatively prime to $S$ on $H^{n-1}(F_+^*, \Meas(F_R \times (\A\!_f^R)^*/U^S, K))$. 

Now let $\lambda'$ be another auxiliary cyclic prime ideal as above.  One can carry out the above construction of the smoothed cocycle $\wmu_{\lambda}$ not only for the single primes $\lambda, \lambda'$, but more generally for a finite set of such primes (see \cite{ds}).
One can easily see that the image of $\wkappa_{\lambda,\lambda'}$ under the canonical map 
\[
H^{n-1}((F_+^{\l, \l'})^*, \Meas(F_R \times (\A\!_f^{R\cup S_{\l, \l'}})^*/U^{S\cup S_{\l, \l'}}, K)) \cong H^{n-1}(F_+^*, \Meas(F_R \times (\A\!_f^R)^*/U^S, K))
\]
is equal to the image of left and right hand side of (\ref{e:changelambda}). 

By class field theory we can view our character $\chi$ as a character of the idele class group 
\[ \chi\colon (\A_f)^*/F_+^*U^S\lra \cO_K^*. \] By assumption the local components $\chi_{\fp}\colon F_{\fp}^* \lra K^*$ of $\chi$ are trivial for all $\fp\in R$, so we may omit them, i.e.\ we view $\chi$ as a character 
\[ \chi^R\colon (\A\!_f^R)^*/F_+^*U^S\lra K^*. \]
The character $\chi^R$ can thus be viewed as an element 
 \begin{equation} \label{e:chir}
 \chi^R \in H^0(F_+^*, C((\A\!_f^R)^*/U^S, K)). \end{equation} 
 Let $\cF$ denote a finite subset of $(\A\!_f^R)^*/U^R$ that is a fundamental domain for the action of $F_+^*/E_R^*$.  For example, we may choose \[ \cF = \{b_1U^R,   \ldots, b_h U^R\} \] where $b_1, \ldots, b_h\in (\A\!_f^R)^*$ are ideles whose associated fractional $\cO_{F,R}$-ideals are $(\fb_1)_R, \ldots, (\fb_h)_R$. 
 The constant function 1 yields an element $1_\cF \in H^0(E_R^*, C(\cF, \Z))$, which under cap product with $\vartheta$ yields an element
 \begin{equation} \label{e:1fc}
1_{\cF} \cap \vartheta \in H_{n+ r- 1}(E_R^*, C(\cF, \Z)) \cong H_{n+r-1}(F_+^*, C_c((\A\!_f^R)^*/U^R, \Z)).
\end{equation}
The isomorphism in (\ref{e:1fc}) follows from Shapiro's Lemma. We denote the image of $1_{\cF} \cap \vartheta$ under this isomorphism by $\vartheta^R$. By taking the cap product with (\ref{e:chir})  we obtain a class
\begin{equation}
\chi^R \cap \vartheta^R\in H_{n+r-1}(F_+^*, C_c((\A\!_f^R)^*/U^S, K)).
\label{e:intcohom2a}
\end{equation}

The first cap-product in (\ref{e:intcohom2a}) 
is induced by the canonical pairing
\begin{align*}
C((\A\!_f^R)^*/U^S, K) \times C_c((\A\!_f^R)^*/U^R, \Z) &\lra C_c((\A\!_f^R)^*/U^S,K), \\
  (f,g) & \longmapsto f\cdot g.
\end{align*}
By taking the cap-product of the homology class (\ref{e:intcohom2a}) with $c_{\l, J}, c_o \in H^r( F_+^*, C_c(F_R, K))$ we obtain classes 
\[
c_{\l, J}\cap (\chi^R \cap \vartheta^R), \, c_o\cap (\chi^R \cap \vartheta^R)\in H_{n-1}(F_+^*, C_c(F_R\times (\A\!_f^R)^*/U^S, K)).
\]
Similar to \eqref{e:cap1} we have a canonical pairing 
\begin{equation}
\cap: H^{n-1}(F_+^*, \Meas(F_R \times (\A\!_f^R)^*/U^S, K))\times H_{n-1}(F_+^*, C_c(F_R\times (\A\!_f^R)^*/U^S, K)) \lra K
\label{e:intcohom2b}
\end{equation}
induced by $p$-adic integration. Tracing through the definitions, one sees that 
\begin{equation}
\label{e:regadelic}
c \cap (\kappa_{\chi, \lambda} \cap \vartheta) = c\cap (\wkappa_{\lambda} \cap (\chi^R \cap \vartheta^R))
\end{equation}
for $c= c_{\l, J}$ or $c=c_0$ and in particular
\begin{equation}
\label{e:regadelic2}
\sR_p(\chi)_{J, \an} = (-1)^r \frac{c_{\l, J}\cap (\wkappa_{\lambda} \cap (\chi^R \cap \vartheta^R))}{c_o\cap (\wkappa_{\lambda}\cap (\chi^R \cap \vartheta^R))}.
\end{equation}
Note that the pairing (\ref{e:intcohom2b}) is $\sI^S$-equivariant. It follows that for $\fa\in \sI^S$ and any class $c\in H^r( F_+^*, C_c(F_R, K))$ we have 
\begin{align*}
c\cap ((\fa^{-1} \cdot \wkappa_{\lambda}) \cap (\chi^R \cap \vartheta^R)) &=c \cap (\wkappa_{\lambda} \cap (\fa \cdot (\chi^R \cap \vartheta^R)))\\
&= \chi(\fa) \, c\cap (\wkappa_{\lambda} \cap (\chi^R \cap \vartheta^R)).
\end{align*}
Consequently, the fraction on the right hand side of (\ref{e:regadelic2}) does not change if we  replace $\wkappa_{\lambda}$ by $\wkappa_{\lambda} - \l' (\lambda'^{-1} \cdot \wkappa_{\lambda})$. Hence by (\ref{e:changelambda}) the constant $\sR_p(\chi)_{J, \an}$ does not depend on the choice of the auxiliary prime $\lambda$.
\end{proof}

\subsection{Alternate formulation}

We give now a slightly different description of $\sR_p(\chi)_{J, \an}$ that will be used in section \ref{ss:diag}. In the following $J$ denotes a fixed non-empty subset of $R$. We set $s= \# J$ and label the elements of $J$ by $\fp_1, \fp_2, \dots, \fp_s$. Let $E_J^*$ be the group of totally positive $J$-units in $F$. 

Given a fractional ideal $\fb$ of $F$ with $(\fb, S) = 1$, a compact open subset $U$ of $F_J$ and a union of simplicial cones $C$ in $(\R^{>0})^n$, we consider the Shintani zeta function
\begin{equation} \label{e:szf}
 \zeta(C, \fb, U, s):=  \N\fb^{-s}\sum_{\alpha} \N\alpha^{-s}, 
 \end{equation}
where $\alpha$ ranges over elements of $\fb^{-1}_J= \fb^{-1}\otimes_{\cO_F} \cO_{F, J}$ satisfying the conditions $(\alpha, S-R) = 1$,  $\alpha \equiv 1 \pmod{\ff}$, $\alpha \in U$, and $\alpha \in C$. Using the auxiliary prime $\lambda \subset \cO_F$ satisfying the conditions stated in \S\ref{s:eis}, we define
\begin{equation} \label{e:szf2} \zeta_\lambda(C, \fb, U, s):= \zeta(C, \fb, U, s) - \l^{1-s} \zeta(C, \fb \lambda^{-1}, U, s). 
\end{equation}
Let $\ff$ be the conductor of the extension $H/F$.  By $E(\ff)$ (respectively $E(\ff)_J$) we denote the group of totally positive units (respectively totally positive $J$-units) congruent to $1$ (mod $\ff$). For $x_1, \dotsc, x_n \in E(\ff)_J$ and compact open $U \subset F_J$ we put
\begin{equation}
\nu_{\fb, \lambda}^J(x_1, \dotsc, x_n)(U) := \sgn(x) \zeta_\lambda( C^*(x_1, \dotsc, x_n), \fb, U, 0). 
\label{e:nupartf}
\end{equation}
Then $\nu_{\fb, \lambda}^J$ is again a homogeneous $(n-1)$-cocycle on $E(\ff)_J$ with values in the space of $\Z$-distribution on $F_J$, hence defines a class
\begin{equation}
\omega_{\ff, \fb, \lambda}^J := [\nu_{\fb, \lambda}^J]\in H^{n-1}(E(\ff)_J, \Meas(F_J, K)).
\label{e:nupartclass}
\end{equation}
Let $G_\ff$ denote the narrow ray class group of $F$ of conductor $\ff$. We also consider the following variant of the Eisenstein cocycle (\ref{e:kapchi})
\begin{equation}
\label{e:nuchif}
\omega_{\chi, \lambda}^J = \sum_{[\fb] \in G_\ff}  \chi(\fb) \omega_{\ff, \fb, \lambda}^J
\end{equation}
where the sum ranges over a system of representatives of $G_{\ff}$.
We also define the following classes in $H^r( F_J^*, C_c(F_J, K))$
\[ c_{o}^J= c_{o_{\fp_1}}  \cup \cdots \cup c_{o_{\fp_s}}, \qquad c_{\l}^J = c_{\l_{\fp_1}} \cup \cdots \cup c_{\l_{\fp_s}}.
\]

\begin{prop} 
\label{p:Jvariant}
Let $\vartheta'\in H_{n+s-1}(E(\ff)_J, \Z)$ be a generator. Then we have
\begin{equation}
\label{e:Jvariant}\sR_p(\chi)_{J, \an} = (-1)^{\#J} \, \frac{c_{\l}^J \cap (\omega_{\chi, \lambda}^J \cap \vartheta')
 }{c_{o}^J \cap (\omega_{\chi, \lambda}^J \cap \vartheta')} . 
 \end{equation}
In fact  the numerator and denominator of the right hand sides of (\ref{e:rpandef}) and of (\ref{e:Jvariant}) coincide up to the same sign. 
\end{prop}

\begin{proof} As in the proof of Prop.\ \ref{p:kappa} it is best to work within the adelic framework. 
We put $J' = R-J$, $S' = S-J'$ and label the elements of $J'$ by $\fp_{s+1}, \dots, \fp_r$. 

By replacing everywhere $R$ by $J$ in the definition of (\ref{e:wkappa}) and (\ref{e:1fc}) we obtain classes
\begin{equation*}
\wkappa_{\lambda}^J \in H^{n-1}(F_+^*, \Meas(F_J \times (\A\!_f^J)^*/U^{S'}, K)).
\end{equation*}
and 
\begin{equation*} 
\vartheta^J \in H_{n+s-1}(F_+^*, C_c((\A\!_f^J)^*/U^J, \Z)).
\end{equation*}
We claim that 
\begin{eqnarray}
\label{e:Jvariant1a}
c_{\l, J} \cap (\kappa_{\chi, \lambda} \cap \vartheta) = \pm c_{\l}^J \cap (\wkappa_{\lambda}^J \cap (\chi^J \cap \vartheta^J)), \\
\label{e:Jvariant1b} c_{o} \cap (\kappa_{\chi, \lambda} \cap \vartheta) =    \pm c_o^J\cap (\wkappa_{\lambda}^J \cap (\chi^J \cap \vartheta^J)).
\end{eqnarray}
For this it suffices to show by (\ref{e:regadelic}) that
\begin{equation}
\label{e:Jvariant2}
c_o^{J'}\cap (\wkappa_{\lambda}^R \cap (\chi^R \cap \vartheta^R)) = \pm \wkappa_{\lambda}^J \cap (\chi^J \cap \vartheta^J). 
\end{equation}
since $c_{\l, J}= c_{\l}^J\cup c_o^{J'}$ and $c_o = c_o^J\cap c_o^{J'}$.

To prove (\ref{e:Jvariant2}) we introduce maps
\begin{eqnarray*}
&& \delta^J: C_c((\A\!_f^J)^*/U^{S'}, K) \lra  C_c(F_{J'} \times (\A\!_f^R)^*/U^S, K)\\
&& \delta: C_c(F_J \times (\A\!_f^J)^*/U^{S'}, K)  \lra  C_c(F_R \times (\A\!_f^R)^*/U^S, K)
\end{eqnarray*} 
that have as local components at the places $\fp \in J'$ the maps 
\[
\delta_{\fp}: C_c(F_{\fp}^*/U_{\fp}, \Z) \lra C_c(F_{\fp}, \Z), 1_{xU_{\fp}} \mapsto 1_{x\cO_{\fp}}
\] 
introduced in \cite[Remark 3.2]{ds} and that are the identity at all other places. More precisely we have canonical isomorphisms
\begin{eqnarray*}
&& C_c((\A\!_f^J)^*/U^{S'}, K) \cong \bigotimes_{i=s+1}^r C_c(F_{\fp_i}^*/U_{\fp_i}, \Z)\otimes C_c((\A\!_f^R)^*/U^{S}, K)\\
&& C_c(F_J \times (\A\!_f^J)^*/U^{S'}, K) \cong \bigotimes_{i=s+1}^r C_c(F_{\fp_i}^*/U_{\fp_i}, \Z)\otimes C_c(F_J \times (\A\!_f^R)^*/U^{S}, K)\\
\end{eqnarray*}
and canonical monomorphisms
\begin{eqnarray}
&& \bigotimes_{i=s+1}^r C_c(F_{\fp_i}, \Z)\otimes C_c((\A\!_f^R)^*/U^{S}, K)
\hookrightarrow C_c(F_{J'} \times (\A\!_f^R)^*/U^S, K) \label{e:tens1}\\
&& \bigotimes_{i=s+1}^r C_c(F_{\fp_i}, \Z)\otimes C_c(F_J \times (\A\!_f^R)^*/U^{S}, K)\hookrightarrow 
C_c(F_R \times (\A\!_f^R)^*/U^S, K) \label{e:tens2}
\end{eqnarray}
and we define $\delta^J$ and $\delta$ as the composite of 
\[
\left(\bigotimes_{i=s+1}^r  \delta_{\fp_i}\right) \otimes \Id_{C_c((\A\!_f^R)^*/U^{S}}, K)\qquad \mbox{and} \qquad \left(\bigotimes_{i=s+1}^r  \delta_{\fp_i}\right) \otimes \Id_{C_c(F_J \times (\A\!_f^R)^*/U^{S}, K)}
\] 
with (\ref{e:tens1}) and (\ref{e:tens2}) respectively. 

Dualizing $\delta$ yields 
\[
\Delta:  \Meas(F_R \times (\A\!_f^R)^*/U^S, K)\lra  \Meas(F_J \times (\A\!_f^J)^*/U^{S'}, K).
\]
By tracing through the definitions one sees that $\wkappa_{\lambda}^J$ is the image of $\wkappa_{\lambda}^R$ under the induced homomorphism
\[
\Delta_*: H^{n-1}(F_+^*, \Meas(F_R \times (\A\!_f^R)^*/U^S, K))\lra H^{n-1}(F_+^*, \Meas(F_J \times (\A\!_f^J)^*/U^{S'}, K)).
\]
On the other hand by \cite[Lemma 3.5]{ds} the following equation holds in the homology group $H_{n+s-1}(F_+^*C_c(F_{J'} \times (\A\!_f^R)^*/U^S, K))$:
\[
(\delta^J)_*(\chi^J \cap \vartheta^J) = \pm c_o^{J'}\cap (\chi^R \cap \vartheta^R).
\]
We conclude
\begin{eqnarray*}
& c_o^{J'}\cap (\wkappa_{\lambda}^R \cap (\chi^R \cap \vartheta^R)) = \pm \wkappa_{\lambda}^R \cap (\delta^J)_*(\chi^J \cap \vartheta^J) \\
& = \pm \Delta_*(\wkappa_{\lambda}^R) \cap (\chi^J \cap \vartheta^J) = \pm \wkappa_{\lambda}^J \cap (\chi^J \cap \vartheta^J).
\end{eqnarray*}
Having established (\ref{e:Jvariant2}) it remains to show that numerator and denominator of the right hand side of (\ref{e:Jvariant}) are equal to the right hand side of 
(\ref{e:Jvariant1a}) and (\ref{e:Jvariant1b}) respectively. Let 
$\ff = \prod_{\fq\nmid \infty} \fq^{n_{\fq}}$ be the prime decomposition of $\ff$ and put
\[
U(\ff)^J = \prod_{\fq\nmid\infty, \fq\not\in J } U_\fq^{(n_\fq)} 
\]
and 
\[
\Gamma(\ff) = F_+^* \cap \prod_{\fq\mid \ff} U_\fq^{(n_\fq)} = \{x\in F_+^*\mid \ord_{\fq}(x-1) \ge n_\fq\,\,\forall\,\,\fq\mid \ff\}.
\]
Here as usual we have set
\[
U_{\fq}^{(n_\fq)} = \left\{\begin{array}{cc} \cO_v^*& \mbox{if $n_\fq =0$;}\\
1 + \fq^{n_{\fq}} \cO_{\fq} & \mbox{if $n_\fq>0$, i.e.\ if $\fq\in S_{\ram}$.} 
\end{array}\right.
\]
Let $\pi: F_J \times (\A\!_f^J)^*/U^{S'} \to F_J \times (\A\!_f^J)^*/U(\ff)^J$ denote the canonical projection. It induces a map
\[
C_c(F_J \times (\A\!_f^J)^*/U(\ff)^J, \Z) \to C_c(F_J \times (\A\!_f^J)^*/U^S, \Z), \qquad f\mapsto f\circ \pi
\]
hence by dualising a homomorphism
\begin{equation}
\label{e:imagemeas}
\pi: \Meas(F_J \times (\A\!_f^J)^*/U^{S'}, K) \to \Meas(F_J \times (\A\!_f^J)^*/U(\ff)^J, K).
\end{equation}
We denote by $\womega_{\ff, \lambda}^J$ the image of $\wkappa_{\lambda}^J$ under \begin{eqnarray*}
H^{n-1}(F_+^*, \Meas(F_J \times (\A\!_f^J)^*/U^{S'}, K)) & \lra & H^{n-1}(F_+^*, \Meas(F_J \times (\A\!_f^J)^*/U(\ff)^J, K))\\
& \stackrel{\cong}{\lra} & H^{n-1}(\Gamma(\ff), \Meas(F_J \times (\A\!_f^{J\cup S_{\ram}})^*/U^{J\cup S_{\ram}}, K))
\end{eqnarray*}
where the first arrow is induced by (\ref{e:imagemeas}) and the second by weak approximation and Shapiro's Lemma. We also denote the image of $\vartheta^J$ under 
\begin{eqnarray*}
H_{n+s-1}(F_+^*, C_c((\A\!_f^J)^*/U^J, \Z)) & \lra & H_{n+s-1}(F_+^*, C_c((\A\!_f^J)^*/U(\ff)^J, \Z))\\
& \stackrel{\cong}{\lra} & H_{n+s-1}(\Gamma(\ff), C_c((\A\!_f^{J\cup S_{\ram}})^*/U^{J\cup S_{\ram}}, \Z))
\end{eqnarray*}
by $\vartheta_{\ff}^J$. Here the first map is induced by the natural projection $(\A\!_f^J)^*/U(\ff)^J\to (\A\!_f^J)^*/U^J$ and the second again by Shapiro's Lemma. 

Moreover we view the character $\chi^J$ as an element of 
\begin{equation*} 
H^0(F_+^*, C((\A\!_f^J)^*/U(\ff)^J, K))\cong H^0(\Gamma(\ff), C((\A\!_f^{J\cup S_{\ram}})^*/U^{J\cup S_{\ram}}, \Z))
\end{equation*}
so that we have 
\begin{equation}
\label{e:Jvariant3}
c\cap (\wkappa_{\lambda}^J \cap (\chi^J \cap \vartheta^J)) = c\cap (\womega_{\ff, \lambda}^J \cap (\chi^J \cap \vartheta_{\ff}^J))
\end{equation}
for $c= c_{\l}^J$ of $c_o^J$. Note that the cohomology groups
\[ 
H^{n-1}(\Gamma(\ff), \Meas(F_J \times (\A\!_f^{J\cup S_{\ram}})^*/U^{J\cup S_{\ram}}, K)),\qquad H^0(\Gamma(\ff), C((\A\!_f^{J\cup S_{\ram}})^*/U^{J\cup S_{\ram}}, \Z))
\]
and the homology group 
\[
 H_{n+s-1}(\Gamma(\ff), C_c((\A\!_f^{J\cup S_{\ram}})^*/U^{J\cup S_{\ram}}, \Z)), \\
\]
all carry a natural $G_{\ff} \cong (\A\!_f^{J\cup S_{\ram}})^*/\Gamma(\ff)U^{J\cup S_{\ram}}$-action. Consider the homomorphisms
\begin{equation}
\label{e:shapiro3}
H_{n+s-1}(E(\ff)_J, \Z) \to H_{n+s-1}(\Gamma(\ff), C_c((\A\!_f^{J\cup S_{\ram}})^*/U^{J\cup S_{\ram}}, \Z))
\end{equation}
induced by the $E(\ff)_J$-equivariant map $\Z\to C_c((\A\!_f^{J\cup S_{\ram}})^*/U^{J\cup S_{\ram}}, \Z)$ sending $1$ to the characteristic function of $U^{J\cup S_{\ram}}$ and by the inclusion $E(\ff)_J\hookrightarrow \Gamma(\ff)$. By abuse of notation we denote the image of the generator $\vartheta'\in H_{n+s-1}(E(\ff)_J, \Z)$ under (\ref{e:shapiro3}) by $\vartheta'$ as well. It is easy to see that with the correct choice of sign of $\vartheta'$ we have
\begin{equation}
\label{e:raydecomp}
\vartheta_{\ff}^J = \sum_{[\fb] \in G_\ff} [\fb]\cdot \vartheta'
\end{equation}
Note that $\chi^J \cap \vartheta'$ is independent of $\chi$. In fact if $\epsilon\in H^0(\Gamma(\ff), C((\A\!_f^{J\cup S_{\ram}})^*/U^{J\cup S_{\ram}}, \Z))$ denotes the characteristic function of $\Gamma(\ff) U^{J\cup S_{\ram}}/U^{J\cup S_{\ram}}$ then we have 
\begin{equation*}
\chi^J \cap \vartheta' =  \epsilon \cap \vartheta
\end{equation*}
Thus for $c= c_{\l}^J$ or $c= c_o^J$ it follows that
\begin{align*}
 c\cap (\womega_{\ff, \lambda}^J \cap (\chi^J \cap \vartheta_{\ff}^J)) &= c \cap \left(\sum_{[\fb] \in G_\ff} \womega_{\ff, \lambda}^J \cap (\chi^J \cap ( [\fb]^{-1} \cdot \vartheta'))\right)\\
& = c\cap \left(\sum_{[\fb] \in G_\ff} ([\fb] \cdot \womega_{\ff, \lambda}^J) \cap (([\fb] \cdot \chi^J) \cap \vartheta')\right) \\
&= \sum_{[\fb] \in G_\ff} \chi(\fb)\, c\cap (([\fb] \cdot \womega_{\ff, \lambda}^J) \cap (\epsilon \cap \vartheta')).
\end{align*}
Passing back from the idele- to ideal-theoretic language we get
\begin{equation*}
\label{e:Jvariant4a}
c \cap (([\fb] \cdot \womega_{\ff, \lambda}^J) \cap (\epsilon \cap \vartheta')) = c \cap (\omega_{\ff, \fb, \lambda}^J\cap \vartheta').
\end{equation*}
Multiplying (\ref{e:Jvariant3}) with $\chi(\fb)$ and summing over the set of representatives $\fb$ of $G_{\ff}$ yields 
\begin{equation}
\label{e:Jvariant4}
c \cap (\womega_{\ff, \lambda}^J \cap (\chi^J \cap \vartheta_{\ff}^J)) = c \cap (\omega_{\chi, \lambda}^J \cap \vartheta').
\end{equation}
Finally, by combining (\ref{e:regadelic}), (\ref{e:Jvariant1a}), (\ref{e:Jvariant1b}), (\ref{e:Jvariant3}) and (\ref{e:Jvariant4}) the assertion follows.
\end{proof}

\section{Evidence}

\subsection{The full Gross regulator}

The following result implies that Conjecture~\ref{c:main2} for $J=R$ is equivalent to part (2) of Conjecture~\ref{c:gross}.  Hence by \cite{dkv}, Conjecture~\ref{c:main2} holds unconditionally in this case. 

\begin{theorem}   \label{t:jr}
For $J = R$, we have \[ \sR_p(\chi)_{J, \an} =  \frac{L_p^{(r)}(\chi, 0)}{r! L(\chi, 0) \prod_{\fp \in R_1} (1 - \chi(\fp))}, \]
and hence Conjecture~\ref{c:main2} holds for $J=R$.
\end{theorem}
 
Before proving Theorem~\ref{t:jr}, we must first relate the Eisenstein cocycle to the $p$-adic $L$-function of $\chi$.
For this,  we first need to extend the definition of $\mu_{\chi,\fb, \lambda}$ so that it involves all primes of $F$ above $p$. Put $R'= R_0\cup R_1$ and $\cO_{R'}^* = \prod_{\fp\in R'} \cO_{\fp}^*$. Instead of considering only compact open subsets of $F_R$ we may consider more generally compact open subsets $U$ of $F_R\times \cO_{R'}^*$ in the definition of $L(C, \chi, \fb, U, s)$.  As before we obtain a homogeneous ($n-1$)-cocycle  $\wmu_{\chi, \fb, \lambda}$ with values in $\Meas(F_R\times \cO_{R'}^*, K)$ that is mapped to $[\mu_{\chi, \fb,\lambda}]$
under the map 
\[
H^{n-1}(E_R^*, \Meas(F_R\times \cO_{R'}^*, K))\lra H^{n-1}(E_R^*, \Meas(F_R, K))
\]
induced by $\pi_*: \Meas(F_R\times \cO_{R'}^*, K)\to \Meas(F_R, K)$, where 
$\pi: F_R\times \cO_{R'}^*\to F_R$ denotes the projection. 

Let $F_{\infty}/F$ be the cyclotomic $\Z_p$-extension of $F$, and let 
\[ \Gamma=\Gal(F_{\infty}/F), \qquad \Lambda=\Z_p[\![\Gamma]\!]\otimes_{\Z_p} K. \] The action of $\Gamma$ on $p$-power roots of unity allows us to view $\Gamma$ as a subgroup of $1+2p\Z_p$. For $\gamma\in \Gamma$ we denote by $\iota(\gamma)$ the corresponding element in $\Lambda$. We view the reciprocity map of class field theory for the extension $F_{\infty}/F$ as a map 
\begin{equation}
\rec: F_p^*\times \sI^p = \prod_{\fp\mid p} F_{\fp}^*\times \sI^p \lra \Gamma 
\label{e:rec}
\end{equation} 
where $\sI^p= \sI^{S_p}$ denotes the group of fractional ideals of $F$ that are relatively prime to $p$. The restriction of (\ref{e:rec}) to $\sI^p$ will be denoted by 
\begin{equation}
\sI^p \lra \Gamma,\,\, \fa \mapsto \gamma_{\fa}
\label{e:recartin}
\end{equation} 
and to $F_R^*\times \cO_{R'}^*\subseteq F_p^*$ by 
\begin{equation}
\rec_{F_{\infty}/F, p}: F_R^*\times \cO_{R'}^*\lra \Gamma.
\label{e:recp}
\end{equation} 
We can view $\iota \circ\rec_{F_{\infty}/F, p}$ as an element 
\begin{equation} \label{e:iotarec}
\iota \circ\rec_{F_{\infty}/F, p} \in H^0(E_R^*, C(F_R^*\times \cO_{R'}^*, \Lambda)).
\end{equation} 

Let $\cF$ denote a compact open subset of $F_R^*\times \cO_{R'}^*$ that is stable under the group of totally positive units of $F$ and such that $F_R^*\times \cO_{R'}^*$ is the disjoint union of the cosets $x\cF$ where $x$ runs through a system of representatives of $E_R^*/E^*$.  As before let $\vartheta_0$ denote a generator of $H_{n-1}(E^*, \Z)$.  As in (\ref{e:1fc}), we can consider the element 
\begin{equation} \label{e:2fc}
1_{\cF} \cap \vartheta_0 \in H_{n-1}(E_+, C(\cF, \Z)) \cong H_{n-1}(E_R^*, C_c(F_R^*\times \cO_{R'}^*,\Z)),
\end{equation}
where the isomorphism is by Shapiro's Lemma since 
 \[ \Ind_{E^*}^{E_R^*} C(\cF, \Lambda) \cong C(F_R^*\times \cO_{R'}^*, \Lambda). \]

Taking the cap product of (\ref{e:iotarec}) and (\ref{e:2fc}) yields a 
 class 
\begin{equation}
\rho = (\iota\circ \rec_{F_{\infty}/F, p} )\cap (1_{\cF} \cap \vartheta_0)\in H_{n-1}(E_R^*, C_c(F_R\times \cO_{R'}^*,\Lambda)).
\label{e:intcohom2}
\end{equation}

The first cap-product is induced by the pairing
\begin{align*}
C(F_R^*\times \cO_{R'}^*,\Lambda) \times C_c(F_R^*\times \cO_{R'}^*,\Z) & \lra C_c(F_R\times \cO_{R'}^*,\Lambda), \\
 (f,g) &\longmapsto (f\cdot g)_!
\end{align*}
where the subscript $!$ denotes ``extension by zero". Similar to (\ref{e:cap1}) we have a cap-product pairing
\[
H^{n-1}(E_R^*, \Meas(F_R\times \cO_{R'}^*, K)) \times H_{n-1}(E_R^*, C_c(F_R\times \cO_{R'}^*,\Lambda))\lra \Lambda,
\]
so we can consider $[\wmu_{\chi, \fb, \lambda}] \cap \rho \in \Lambda$. To link this element of the Iwasawa algebra to the $p$-adic $L$-function $L_p(\chi,s)$ we recall that there exists a canonical homomorphism 
\begin{equation}
\Xi: \Lambda \lra C^{\an}(\Z_p, K)
\label{e:iwasfunc}
\end{equation} 
characterized by \[ \Xi(\iota(\gamma))(s) = \gamma^s:=\exp_p(s\log_p(\gamma)) \] for all $\gamma\in \Gamma$ and $s\in \Z_p$. Here $C^{\an}(\Z_p, K)$ is the $K$-algebra of locally analytic maps $f\colon\Z_p\lra K$. 

\begin{prop}
We have 
\begin{equation}
\Xi\left(\sum_{i=1}^h \chi(\fb_i) \iota(\gamma_{\fb_i}) [\wmu_{\chi, \fb_i, \lambda}] \cap \rho\right)(s) \, =\, (1-\chi(\lambda)\gamma_{\lambda}^s \l) L_p(\chi,s).
\label{e:cap2}
\end{equation}
\end{prop}
\begin{proof}
This formula is a variant of \cite[Prop.\ 5.6]{ds} and the proof there carries over.  In fact, as we now explain,
the present result can be deduced from the statement of {\em loc.\ cit.}
It is well known that the $p$-adic $L$-function interpolates to an element of $\Lambda$, i.e.\ that there exists a unique element
\[ \sL_{p, \lambda}(\chi) \in \Lambda \] such that 
\[ \Xi(\sL_{p, \lambda}(\chi)) =  (1-\chi(\lambda)\gamma_{\lambda}^s \l) L_p(\chi,s). \]
We must show that the expression in parenthesis on the left side of (\ref{e:cap2}) is equal to $\sL_{p, \lambda}(\chi)$, and for this if suffices to show that they agree under application of the dense set of homomorphisms 
\[ \tilde{\psi} \colon \Lambda \lra \C_p^* \]
induced by $p$-power conductor Dirichlet characters $\psi: \Gamma \lra \mu_{p^{\infty}}$. (These homomorphisms are ``dense" in the sense that the intersection of their kernels in $\Lambda$ is trivial.)  In other words, we must show that
\begin{equation}
\label{e:lpinterp}
 \tilde{\psi}\left(\sum_{i=1}^h \chi(\fb_i) \iota(\gamma_{\fb_i}) [\wmu_{\chi, \fb_i, \lambda}] \cap \rho\right) = 
(1-\chi\psi(\lambda) \l) L_S(\chi \psi,0).
\end{equation}
Now if we let $K$ be the fixed field of $\chi\psi$, set $k=0$, and apply the character $\chi\psi$ to the equation
in \cite[Prop.\ 5.6]{ds}, then we obtain exactly (\ref{e:lpinterp}).
\end{proof}

\begin{proof}[Proof of Theorem~\ref{t:jr}]
By Prop.\ \ref{p:denom} it suffices to show
\begin{equation}
c_{\l, J} \cap (\kappa_\chi \cap \vartheta) =  (-1)^{rn} (1 - \chi(\lambda) \l) L_p^{(r)}(\chi, 0)/r!.  \label{e:cellcap}
 \end{equation}

In order to study the leading term of (\ref{e:cap2}) at $s=0$, we consider the homomorphism of $K$-algebras
\begin{align*}
\Ta_{\le r}: C^{\an}(\Z_p, K) \, &\lra \, K[X]/(X^{r+1}),\\
 f&\longmapsto \Ta_{\le r} f = \sum_{k=0}^r \frac{f^{(k)}(0)}{k!} \barX^k
\end{align*}
and the composite 
\begin{equation} \label{e:composite}
\Ta_{\le r} \circ \ \Xi \circ \iota \circ \rec_{F_{\infty}/F, p}: F_R^*\times \cO_{R'}^*\lra  \Gamma \lra \Lambda \lra C^{\an}(\Z_p, K) \lra K[X]/(X^{r+1}).
\end{equation}
Restricting (\ref{e:composite}) to $F_{\fp}^*$ for $\fp\in R$ and reducing  modulo $(\barX^2)$ yields the homomorphism 
\begin{equation} \label{e:lpclass}
F_{\fp}^* \lra K[X]/(X^2), \qquad a\longmapsto 1- \l_{\fp}(a) \barX.
\end{equation}
As before we consider the class
\[
\overline{\rho}:=(\Ta_{\le r}\circ \ \Xi \circ \iota)\cap (1_{\cF} \cap \vartheta_0)\in H^{n-1}(E_R^*, C_c(F_R\times \cO_{R'}^*,K[X]/(X^{r+1}))).
\]
Applying \cite[Prop.\ 3.6]{ds} we obtain  
\begin{equation} \label{e:rhobar}
\overline{\rho} \,=\, (-1)^r c_{\l, J}\cap ((\barX^r 1_{\cO_{R'}^*})\cap \vartheta).
\end{equation}
(To make the connection with the notation in {\em loc.\ cit.}, note that our $\overline{\rho}$ and $1_{\cO_{R'}^*}\cap \vartheta$ correspond to $\kappa$ and $\overline{\kappa}$ there, and that in view of (\ref{e:lpclass}) our $(-1)^r c_{\l, J} \overline{X}^r$ corresponds to $c_{d\chi_1} \cup \cdots \cup c_{d\chi_r}$ there.)
In (\ref{e:rhobar}), the second cap-product lies in $H_{n+r-1}(E_R^*, C(\cO_{R'}^*, K[X]/(X^{r+1})))$. Therefore
\begin{align}
(\Ta_{\le r}\circ \ \Xi)([\wmu_{\chi, \fb, \lambda}] \cap \rho) & =  [\wmu_{\chi, \fb, \lambda}] \cap \overline{\rho} \nonumber \\
&= (-1)^r (-1)^{(n-1)r} \, (c_{\l, J} \cup  (\barX^r 1_{\cO_{R'}^*}) \cup [\wmu_{\chi, \fb, \lambda}]) \cap \vartheta \nonumber \\
 & =  (-1)^{nr}\barX^r \, c_{\l, J} \cap ([\mu_{\chi, \fb, \lambda}] \cap \vartheta). \label{e:tar}
\end{align}
Combining (\ref{e:cap2}) and (\ref{e:tar}), we obtain
\begin{align}
 \Ta_{\le r}((1-\chi(\lambda)\gamma_{\lambda}^s \l) L_p(\chi,s))  &=
 \sum_{i=1}^h \chi(\fb_i) \Ta_{\le r} \circ \ \Xi (\iota(\gamma_{\fb_i}) [\mu_{\chi, \fb_i, \lambda}] \cap \rho) \nonumber \\
&=  
 (-1)^{nr}\barX^r \sum_{i=1}^h \chi(\fb_i) (\Ta_{\le r}\circ \Xi)(\iota(\gamma_{\fb_i})) c_{\l, J} \cap ([\mu_{\chi, \fb, \lambda}] \cap \vartheta)\nonumber \\
&=  (-1)^{nr}\barX^r \sum_{i=1}^h \chi(\fb_i) c_{\l, J} \cap ([\mu_{\chi, \fb_i, \lambda}] \cap \vartheta) \label{e:1mod} \\
&=  (-1)^{nr}\, \barX^r  c_{\l, J} \cap (\kappa_{\chi, \lambda}\cap \vartheta) \label{e:key}
\end{align}
where (\ref{e:1mod}) follows from $(\Ta_{\le r}\circ \ \Xi)(\iota(\gamma_{\fb_i})) \equiv 1$ modulo $\barX$. Since
\[
\Ta_{\le r}(1-\chi(\lambda)\, \gamma_{\lambda}^s \,\l) \equiv 1 - \chi(\lambda) \l  \pmod{\overline{X}}
\]
we see that $\Ta_{\le r}(1-\chi(\lambda)\, \gamma_{\lambda}^s\, \l)$ is a unit in $K[X]/(X^{r+1})$. Hence 
\[ \Ta_{\le r}(L_p(\chi,s))\equiv 0  \pmod{\barX^r} \] and 
\[
\Ta_{\le r}((1-\chi(\lambda)\, \gamma_{\lambda}^s\, \l) L_p(\chi,s)) = (1 - \chi(\lambda) \l) \, L_p^{(r)}(\chi, 0)/r! \,\barX^r.
\]
Together with (\ref{e:key}) we conclude (\ref{e:cellcap}). \end{proof}
 
\begin{remark} \rm We would like to point out that Proposition \ref{p:denom} and formula 
(\ref{e:cellcap}) could be deduced directly from \cite[Corollary 3.19(b)]{spiesshilb} and \cite[Corollary 3.22]{spiesshilb}. However we feel that the framework developed in \cite[\S 3] {ds} is somewhat superior to that of \cite[\S 4]{spiesshilb} and think it is worthwhile to present the application to trivial zeros of $p$-adic $L$-functions here again in some detail.
\end{remark}

\subsection{The diagonal entries}
\label{ss:diag}
 
 We now consider the other extremal case $\#J = 1$.  If $J = \{ \fp \}$, then
\[ \sR_p(\chi)_J =  \sL_{\alg}(\chi)_{ \fp, \fp} = - \frac{\l_\fp(u_{\fp, \chi})}{o_{\fp}(u_{\fp, \chi})}. \]
In this setting, the first named author \cite[Conjecture 3.21]{das} had previously conjectured a formula for the image of $u_{\fp, \chi}$ in $F_\fp^* \otimes K$.
We recall below the definition of this conjectural image, denoted $\cU_{\fp, \chi}$.

\begin{theorem} \label{t:jep}  When $n=2$, Conjecture~\ref{c:main} for $J = \{\fp\}$ is consistent with \cite[Conjecture 3.21]{das}, i.e.\ we have
\[ \sR_p(\chi)_{J, \an} =  - \frac{\l_\fp(\cU_{\fp, \chi})}{o_{\fp}(\cU_{\fp, \chi})}.
\]
\end{theorem}

\begin{remark}  We expect Theorem~\ref{t:jep} to be tractable when $n > 2$ as well, but  we leave this as an open problem.
\end{remark}

Before proving Theorem~\ref{t:jep}, we recall the definition of $\cU_{\fp, \chi}$. We keep the notation of the end of \S \ref{section:mainconj}. Let $\cD = \sum_i a_i C_i$ denote a signed Shintani domain for the action of $E(\ff)$ on $(\R^{>0})^n$, so the characteristic function $1_\cD := \sum_i a_i 1_{C_i}$ satisfies $\sum_{\epsilon \in E(\ff)}1_{\cD}( \epsilon \cdot x) = 1$ for all $x \in (\R^{>0})^n$.
 For each fractional ideal $\fb$ of $F$ relatively prime to $\ff$ and $p$, we will define an element $\cU_\fp(\fb, \lambda, \cD) \in F_\fp^*$ and define
 \begin{equation} \label{e:upcdef}
  \cU_{\fp, \chi} = \sum_{[\fb] \in G_\ff}  \cU_\fp(\fb, \lambda, \cD) \otimes \chi(\fb)/(1-\chi(\lambda)\l), 
  \end{equation}
   where the sum ranges over a set of representatives $\fb$ for $G_\ff$. The independence of $ \cU_{\fp, \chi}$ from the choices of the $\fb$, $\cD$, and $\lambda$ is somewhat subtle and is discussed  in \cite[\S 5]{das}.

We now define $\cU_\fp(\fb, \lambda, \cD).$ 
Let $e$ be the order of $\fp$ in $G_\ff$, and write $\fp^e = (\pi)$ where $\pi$ is totally positive and $\pi \equiv 1 \pmod{\ff}$.  
For a compact open subset of $F_{\fp}$ we define 
 \[ 
\nu_{\fb, \lambda, \cD}(U) := \sum_i a_i \zeta_\lambda(\fb, C_i, U, 0), \qquad \text{ where } \cD = \sum a_i C_i. 
\]
where $\zeta_\lambda(\fb, C_i, U, s)$ denotes the function (\ref{e:szf2}). 

Our assumptions on $\lambda$ imply $\nu_{\fb, \lambda, \cD}(U)\in \Z$ (see \cite[Proposition 3.12]{das}). The main contribution to the definition of $\cU_\fp(\fb, \lambda, \cD)$ is a multiplicative integral defined analogously to (\ref{e:intdef}), but with Riemann {\em products} instead of sums:
\[ \mint_{\cO_\fp - \pi \cO_\fp} x \ d\nu_{\fb, \lambda, \cD}(x) := \lim_{|| \cV || \rightarrow 0} \prod_{V \in \cV} t_V^{\nu_{\fb, \lambda, \cD}(V)},
\]
as $\cV = \{V\}$ ranges over uniformly finer covers of $\cO_\fp - \pi \cO_\fp$ and $t_V \in V$.

The element $\cU_\fp(\fb, \lambda, \cD) \in F_\fp^*$ is defined as the product of this multiplicative integral with a certain global unit in $F$ and a power of $\pi$.
Given a formal linear combination of simplicial cones $\cD = \sum a_i C_i$ and a totally positive $x \in F^*$, we define $x\cD = \sum a_i \cdot x  C_i$, with characteristic function $1_{x \cD}(y) = \sum a_i 1_{C_i}(x^{-1} y)$.  Given two such formal linear combinations, we define their intersection as the formal linear combination whose characteristic function is the product:
\[ 1_{\cD \cap \cD'} := 1_{\cD} \cdot 1_{\cD'}. \]
With these notations, we define
\begin{equation} \label{e:edef}
 \epsilon_{\fb, \lambda, \cD, \pi} := \prod_{\epsilon \in E(\ff)} \epsilon^{\nu_{\fb, \lambda, \epsilon \cD \cap \pi^{-1} \cD}(\cO_\fp)}. 
 \end{equation}
One easily shows that there are only finitely many $\epsilon$ for which the exponent in (\ref{e:edef}) is nonzero.
Finally, we define
\[ \cU(\fb, \lambda, \cD) :=  \epsilon_{\fb, \lambda, \cD, \pi} \cdot  \pi^{\nu_{\fb, \lambda, \cD}(\cO_\fp)}  \cdot \mint_{\cO_\fp - \pi \cO_\fp} x \ d\nu_{\fb, \lambda, \cD}(x)\]
and $\cU_{\fp, \chi}$ as in (\ref{e:upcdef}).

We assume now that $n=2$. Recall that we have we have fixed an ordering of the real places of $F$  yielding an embedding $F\subset \R^2$. We choose the generator $\ep$ of $E(\ff)$ so that it lies in the half plane $\{(x,y)\in \R^2\mid x<y\}$. Then we have
\[
\sgn(1, \ep)=1 \qquad \mbox{and} \qquad C^*(1, \ep) = C(1,\ep) \cup C(\ep)
\]
and $\cD = C^*(1, \ep)$ is a Shintani domain for the action of $E(\ff)$ on $(\R^{>0})^2$. For the cocycle (\ref{e:nupartf}) evaluated at the pair $(1, \ep)$ we have 
\begin{equation}
\label{e:jep1}
\nu_{\fb, \lambda}^J(1, \ep) = \nu_{\fb, \lambda, \cD}
\end{equation}
Now Theorem \ref{t:jep} follows immediately from 

\begin{prop} 
\label{p:partgs}  When $n=2$ and $J = \{\fp\}$ and let $[\fb]\in G_{\ff}$. We have 
\[ 
\l_\fp( \cU_\fp(\fb, \lambda, \cD)) = \pm c_{\l_\fp} \cap (\omega_{\ff, \fb, \lambda}^J\cap \vartheta'),\qquad o_\fp( \cU_\fp(\fb, \lambda, \cD)) = \pm c_{o_\fp} \cap (\omega_{\ff, \fb, \lambda}^J\cap \vartheta').
\]
\end{prop}

\begin{proof} After replacing $\pi$ by $\pi\ep^n$ for some $n\in\Z$ we may assume that $\pi\in \cD$. Since $\epsilon, \pi$ are a $\Z$-basis of $E(\ff)_J$ the cycle
\[
\vartheta' : = [\pi | \epsilon] - [\epsilon| \pi]
\]
is a generator of $H_2(E(\ff)_J, \Z)$. Let $g: F_{\fp}^*\to K$ be any continuous homomorphism (e.g.\ $g= \l_{\fp}$ or $g=o_{\fp}$). The assertion follows from 
\begin{equation}
\label{e:comp}
g( \cU_\fp(\fb, \lambda, \cD)) = c_g \cap (\omega_{\ff, \fb, \lambda}^J\cap \vartheta')
\end{equation}
Since $\pi\in \cD=C^*(1,\ep)$ we have 
\begin{equation}
\label{e:jep2}
\ep^n \cD\cap \pi^{-1}\cD = C^*(\ep^n, \ep^{n+1}) \cap C^*(\pi^{-1}, \pi^{-1}\ep) =\left\{\begin{array}{cc} C^*(1, \ep\pi^{-1})& \mbox{if $n =0$;}\\
C^*(1, \pi^{-1}) & \mbox{if $n=-1$;}\\
\emptyset & \mbox{otherwise.} 
\end{array}\right.
\end{equation}
Since $E(\ff) = \langle \ep \rangle$ and $\sgn(1, \pi^{-1})=-1$ this implies 
\begin{eqnarray}
g(\epsilon_{\fb, \lambda, \cD, \pi}) & = & \left(\sum_{n\in \Z} n \nu_{\fb, \lambda, \epsilon \cD \cap \pi^{-1} \cD}(\cO_\fp)\right) g(\ep)
\label{e:jep3}\\
& = & - \nu_{\fb, \lambda, C^*(1, \pi^{-1})}(\cO_\fp) \cdot g(\ep)
\nonumber\\
& = & \nu_{\fb, \lambda}^J(1, \pi^{-1})(\cO_\fp)\cdot  g(\ep)\nonumber\\
& = & - \nu_{\fb, \lambda}^J(1, \pi)(\pi \cO_\fp)\cdot g(\ep)\nonumber
\end{eqnarray}
The last equality follows from the fact that $x\mapsto \nu_{\chi, \lambda}^J(1,x)$ is an inhomogeneous 1-cocyle on $E(\ff)_J$. 

We will choose as representative of $c_g$ the inhomogeneous 1-cocycle $z: = z_{1_{\pi\cO_{\fp}}, g}$ 
i.e.\ we choose $f= 1_{\pi\cO_{\fp}} = \pi 1_{\cO_\fp}$ in (\ref{e:zgdef}). A simple computation yields
\begin{equation}
\label{e:jep4}
z(\ep) = (\pi 1_{\cO_\fp}) \cdot g(\ep) 
\end{equation}
and 
\begin{equation}
\label{e:jep5}
\pi^{-1} z(\pi) = 1_{\cO_{\fp} - \pi \cO_{\fp}} \cdot g + 1_{\cO_{\fp}} \cdot g(\pi)
\end{equation}
Put $\nu = \nu_{\fb, \lambda}^J$ and $\nu_{\cD} = \nu_{\fb, \lambda, \cD}$ so that $\nu(1,\ep) = \nu_{\cD}$ by (\ref{e:jep1}). Using (\ref{e:jep3}), (\ref{e:jep4}) and (\ref{e:jep5}) we get 
\begin{eqnarray*}
c_g \cap (\omega_{\ff, \fb, \lambda}^J\cap \vartheta') & = & \int_{F_{\fp}} z(\pi)(x) d(\pi \nu(1, \ep))(x) - \int_{F_\fp} z(\ep)(x) d(\ep\nu(1, \pi))(x)\\
& = &  \int_{F_{\fp}} (\pi^{-1}z(\pi))(x) d\nu_{\cD} (x) - \int_{F_\fp} (\ep^{-1} z(\ep))(x) d\nu(1, \pi)(x)\\
& = & \int_{\cO_{\fp} - \pi \cO_{\fp}} g(x) d\nu_{\cD} (x) + \nu_{\cD}(\cO_{\fp}) \cdot g(\pi) - \nu(1, \pi)(\pi \cO_\fp)\cdot g(\ep)\\
& = & g\left( \mint_{\cO_\fp - \pi \cO_\fp} x \ d\nu_{\cD}(x)\right) + g\left(\pi^{\nu_{\cD}(\cO_{\fp})}\right) + g(\epsilon_{\fb, \lambda, \cD, \pi})\\
& = & g( \cU_\fp(\fb, \lambda, \cD))
\end{eqnarray*}
\end{proof}

\begin{remark} \rm A more indirect approach towards Theorem~\ref{t:jep} is as follows. In \cite[\S 6]{ds} we have defined certain elements $\cU'_{\fp}(\fb, \lambda)$ of $F_{\fp}$ in terms of the Eisenstein cocycle. We expect that these elements agree with the elements $\cU_{\fp}(\fb, \lambda, \cD)$.
 It should be much easier to verify Theorem~\ref{t:jep} with $\cU'_{\fp}(\fb, \lambda)$ replacing $\cU_{\fp}(\fb, \lambda, \cD)$ in (\ref{e:upcdef}). On the other hand in \cite{das} and \cite[\S 6]{ds} a list of functorial properties for the elements $\cU_{\fp}(\fb, \lambda, \cD)$ and $\cU'_{\fp}(\fb, \lambda)$ have been established. Since these properties determine the elements uniquely up to a root of unity neither $\l_\fp(\cU_{\fp, \chi})$ nor $o_{\fp}(\cU_{\fp, \chi})$ will change while replacing the elements $\cU_{\fp}(\fb, \lambda, \cD)$ with $\cU'_{\fp}(\fb, \lambda)$ in the definition of $\cU_{\fp, \chi}$. 
\end{remark}

\end{document}